\documentclass[11pt]{amsart}

\usepackage{amssymb,xspace}
\usepackage{graphicx}

\newcommand*{\Ab}{\ensuremath{\mathbb A}\xspace}
\newcommand*{\C}{\ensuremath{\mathbb C}\xspace}
\newcommand*{\E}{\ensuremath{\mathcal E}\xspace}

\newcommand*{\Ll}{\ensuremath{\mathcal L}\xspace}

\newcommand*{\m}{\ensuremath{\mathfrak m}\xspace}

\def\Oo{\ensuremath{\mathcal O}\xspace}
\newcommand*{\PP}{\ensuremath{\mathbb P}\xspace}
\newcommand*{\Pc}{\ensuremath{\mathcal P}\xspace}

\newcommand*{\Ss}{\ensuremath{\mathcal S}\xspace}
\newcommand*{\T}{\ensuremath{\mathcal T}\xspace}

\DeclareMathOperator{\Osc}{Osc}
\DeclareMathOperator{\pr}{pr}

\DeclareMathOperator{\Sym}{Sym}

\let\ph\varphi

\newcommand*{\lst}[3][1]{\ensuremath{#2_{#1}, \ldots, #2_{#3}}\xspace}

\newtheorem{prop}{Proposition}[section]

\newtheorem{lemma}[prop]{Lemma}
\newtheorem{cor}[prop]{Corollary}

\theoremstyle{definition}

\numberwithin{equation}{section}

\author{Serge Lvovski}
\address{National Research University Higher School of Economics,
  Moscow, Russia}
\email{lvovski@gmail.com}



\thanks{The article was prepared within the framework of a subsidy
  granted to the HSE by the Government of the Russian Federation for
  the implementation of the Global Competitiveness Program.}

\title{Some remarks on osculating self-dual varieties}

\begin{document}

\begin{abstract}
Let us say that a curve $C\subset\PP^3$ is osculating self-dual if it is
projectively equivalent to the curve in the dual space $(\PP^3)^*$
whose points are osculating planes to~$C$. Similarly, we say that a
$k$-dimensional subvariety $X\subset\PP^{2k+1}$ is osculating
self-dual if its second osculating space at the general point is a
hyperplane and $X$ is projectively equivalent to the variety in
$(\PP^{2k+1})^*$ whose points are second osculating spaces to $X$.

In this note we show that for each $k\ge 1$ there exist many
osculating self-dual $k$-dimensional subvarieties in $\PP^{2k+1}$. 
\end{abstract}

\maketitle

\section{Introduction}
If $C\subset\PP^n$ is a projective curve (not lying in a hyperplane),
then its \emph{osculating dual} is the curve $C^\vee\subset(\PP^n)^*$
that is closure of the set of (points corresponding to) hyperplanes
osculating to~$C$. For this version of duality, the ``duality
theorem'' $(C^\vee)^\vee=C$ in characteristic~$0$ also holds
(see~\cite[Theorem 5.1]{Piene}).

In this note we show that there exist many curves $C\subset\PP^3$ for
which $C$ and $C^\vee$ are projectively equivalent: there exists a
projective (linear) isomorphism $\PP^3\to(\PP^3)^*$ that takes $C$
to~$C^\vee$. In particular, any smooth projective curve can be
embedded in $\PP^3$ as ``osculating self-dual''.

Analogs of this ``osculating'' duality can be defined for varieties of
higher dimension as well. To wit, if $X\subset\PP^{2k+1}$ is a
$k$-dimensional variety such that its second osculating space at
the general point is a hyperplane, then one may define
$X^\vee\subset(\PP^{2k+1})^*$ as closure of the set of points
corresponding to these osculating hyperplanes; for each $k$, we
construct a large family of $k$-dimensional varieties
$X\subset\PP^{2k+1}$ such that the second osculating hyperplanes at
the general point is a hyperplane and $X^\vee$ is projectively equivalent
to~$X$. 

The proofs are based on the following observation: if a
$k$-dimensional subvariety $X\subset \PP^{2k+1}$ is Legendrian with
respect to a contact structure on $\PP^{2k+1}$ then its second
osculating space at the general point is at most $2k$-dimensional.

For $k=2$, surfaces in $\PP^5$ with four-dimensional general second
osculating plane were studied by ancient Italian geometers
(see~\cite{Segre,Togliatti}). In particular, Togliatti
in~\cite{Togliatti} classifies all non-ruled surfaces $X\subset\PP^5$
for which $\deg X\le6$ and the general second osculating space is a
hyperplane and essentially shows that some of these surfaces are
osculating self-dual.

It should be noted that self-dual curves in $\mathbb P^2$ are much
harder to construct. In particular, all known self-dual plane curves
seem to have genus of normalization~$0$ or~$1$. (In the old
paper~\cite{Hollcroft}, which apparently contains many examples of
self-dual plane curves, a curve $C$ is called self-dual if several
numeric invariants of $C$ and $C^*$ are the same, which is, of course,
a weaker condition than projective equivalence.)

The paper is organized as follows. Section~\ref{sec:cs} is devoted to
contact structures on $\PP^{2n-1}$'s an projectivisations of cotangent
bundles to $\PP^n$'s, in Section~\ref{sec:surfaces} we construct
osculating self-dual curves and varieties, and in
Section\ref{sec:conclusion} we show that there exist osculating
self-dual varieties that cannot be obtained by the main construction
of the paper. Section~\ref{sec:prelim} is devoted to preliminaries.

\subsection*{Acknowledgements}
I would like to thank Alexei Penskoi for attracting my attention to
the paper~\cite{Bryant}, Fyodor Zak and Nikita Kalinin for useful
doscussions, and Jason Starr and Robert Bryant for valuable
consultations at \texttt{mathoverflow.net}.

\section{Notation, conventions, and preliminaries}\label{sec:prelim}

\subsection{Generalities}

The base field is the field \C of complex numbers. 

Two subsets $Y_1\subset\PP^n$, $Y_2\subset\PP^n$ will be called
\emph{projectively equivalent} if there exists a projective (linear)
isomorphism $F\colon\PP^n\to\PP^n$ such that $F(Y_1)=Y_2$.

If \E is a vector space or a vector bundle, then (closed) points of
the projectivisation $\PP(\E)$ are lines in (the fibers of)~$E$, not
hyperplanes. By $\PP^*(\E)$, where \E is a vector bundle, we mean
$\PP(\E^*)$; so, points of $\PP^*(\E)$ are hyperplanes in the fibers
of~\E.

If $L\subset\PP(E)$ is a linear (projective) subspace, then the
uniquely determined linear subspace $\hat L\subset E$ such that
$L=\PP(\hat L)$ is called \emph{deprojectivisation} of~$L$. If
$Y\subset\PP(E)$ is a projective variety then by its
deprojectivisation we mean the subvariety $\hat
Y=\overline{\pi^{-1}(Y)}\subset E$, where $\pi\colon
E\setminus\{0\}\to\PP(E)$ is the canonical projection.

\subsection{Contact structures}
 
A \emph{contact structure} on a smooth variety $X$ is a codimension
$1$ subbundle \Ss of the tangent bundle $\T_X$ satisfying a certain
non-degeneracy condition (a precise definition can be found for
example in~\cite{Kleiman}; see also the sketch of proof of
Proposition~\ref{folklore}). If $X$ is a variety with a contact
structure~\Ss, then the fiber of the vector bundle \Ss at a point $x\in X$
is denoted by $\Ss_x\subset \T_{x,X}$ and called \emph{contact
  hyperplane at $x$}. Contact structures can exist only on
odd-dimensional varieties.

If $X$ is a variety with the contact structure \Ss, then a closed
subvariety $Y\subset X$ is called \emph{integral subvariety} of the
structure~\Ss if $\T_yY\subset\Ss_y$ for each smooth point $y\in
Y$. 

Locally, each contact structure on $X$ can be defined as (closure of)
the family of hyperplanes in tangent spaces that are kernels of a
nonvanishing $1$-form $\omega\in\Gamma(\Omega^1_U)$, where $U\subset X$
is a Zariski open set (globally this $\omega$ is a section of
$\Omega^1_X\otimes(\T_x/\Ss)$). A subvariety $Y\subset X$ is integral
if and only if the restriction of $\omega$ to its smooth part is zero.

If $X$ is a variety with a contact structure, $\dim X=2n-1$, then
integral subvarieties of (maximal possible) dimension $n-1$ are called
\emph{Legendrian subvarieties of $X$} with respect to this contact
structure.

\subsection{Osculating spaces and osculating duality}

If $X\subset\PP^N$ is a projective variety and $x\in X$ is a
non-singular point, one says that a hyperplane $H\subset\PP^n$
\emph{osculates at $x$ to order $s$} if $H\ni x$ and the local
equation of $H\cap X$ in the local ring $\Oo_{x,X}$ lies in $\m_x^{s+1}$,
where $\m_x\subset\Oo_{x,X}$ is the maximal ideal. Analytically this
means the following: if \lst zn are analytic local coordinates on $X$
near $x$ and $f$ is a local equation of $H$ at $x$, then the power
series expansion of $f$ begins with terms of degree $\ge s+1$. A
hyperplane osculates at $x$ to order~$1$ if and only if $H$ is tangent
to $X$ at~$x$.

The intersection of all hyperplanes osculating to $X$ at $x$ to
order~$s$ is called \emph{$s$'th osculating space to $X$ at $x$} and
denoted by $\Osc^s_xX$ (if no hyperplane osculates to order $s$ at
$x$, we assume that $\Osc^s_xX$ coincides with the linear span of
$X$). The space $\Osc^1_xX\subset \PP^N$ is nothing but the embedded
tangent space~$T_xX\subset\PP^n$.

If $X\subset \PP^n$ is a curve that is not contained in a hyperplane,
then for general $x\in X$ one has $\dim\Osc_x^jX=j$ for $1\le j\le
n-1$, and $\Osc_x^{n-1}C$ is exactly the osculating hyperplane as
defined in the introduction.

If $x\in X$ is a non-singular point and $H$ is a tangent hyperplane to
$X$ at~$x$, then the image of the local equation of $H\cap X$ in
$\m_x^2/\m_x^3$ defines (up to a multiplicative constant) an element
of $\Sym^2(\m_x/\m_x^2)=\Sym^2(\T_xX)^*$. All the symmetric bilinear forms
in the tangent space (with all their multiples) corresponding, via the
procedure above, to hyperplanes $H\supset T_xX$, form a linear
subspace of $\Sym^2\T_xX$. This linear space is called \emph{second
  fundamental form of $x$ at $X$}; it will be denoted
$\Phi^2_x(X)$. If we fix local analytic coordinates \lst zn at $x$,
then elements of $\Phi^2_x(X)$ are polynomials of degree $s$ in
\lst{dz}n; abusing the language, we will write them as polynomials in
\lst zn. One has $\dim\Phi^2_x(X)=\dim\Osc_x^2X-\dim X$.

Suppose that $X\subset\PP^{2n-1}$ is a projective variety of
dimension $n-1$ that is not contained in a hyperplane.  The expected
value of $\dim\Osc_x^2X$ for general $x\in X$ is $2n-1$, i.\,e., in
the general case this osculating space coincides with the ambient
$\PP^{2n-1}$. If, however, for general $x$, $\Osc_x^2X$ is a hyperplane, or,
equivalently, $\dim\Phi_x^2(X)=n-1=\dim X$ (if $\dim
X=1$, this is automatic, otherwise it is a non-trivial condition), we denote by
$X^\vee\subset(\PP^{2n-1})^*$ the closure of the set of hyperplanes
$\Osc_x^2(X)$ for general $x\in X$ (in the paper~\cite{Valles} this
variety is denoted by $X^{3\vee}$). The subvariety $X^\vee\subset
(\PP^n)^*$ will be called \emph{osculating dual} to~$X$. If $X$ is
  projectively equivalent to $X^\vee$, we will say that $X$ is
  \emph{osculating self-dual}.

\section{Well-known contact structures on $\PP^{2n-1}$ and
  $\PP^*(\T\PP^n)$}\label{sec:cs}

If $E$ is a vector space of even dimension $2n$ and $B$ is a
non-degenerate skew-symmetric bilinear form on $E$, we define a
contact structure on $\PP^{2n-1}=\PP(E)$ as follows. If $x=(v)$ is a
point of $\PP(E)$, where $v\in E\setminus\{0\}$, then the contact
hyperplane $\Ss_x\subset \T_x\PP(E)$ is $\T_x\PP(v^\perp)$, where
$v^\perp\subset E$ is the skew-orthogonal complement to $v$ with
respect to~$B$. If in coordinates the form $B$ is defined by the
formula
\begin{equation}\label{eq:B}
B((z_0,\ldots,z_{2n}),(w_0,\ldots,w_{2n}))=
\sum_{i=0}^{n-1}(z_{2i}w_{2i+1}-z_{2i+1}w_{2i}),
\end{equation}
then on the affine open set $\{(1:z_1:\ldots:z_{2n-1})\subset\PP^{2n-1}\}$
the contact structure corresponding to $B$ can be defined by the form
\begin{equation}\label{eq:omega}
\omega=dz_1+\sum_{i=1}^{n-1}(z_{2i}dz_{2i+1}-z_{2i+1}dz_i).
\end{equation}
Since any two non-degenerate skew-symmetric forms are equivalent, any
two contact structures on $\PP^{2n-1}$ obtained by the above
construction are mapped to each other by a projective automorphism of
$\PP^{2n-1}$.

Actually, there are no other contact structures on projective
spaces. The following proposition seems to belong to folklore.

\begin{prop}\label{folklore}    
Any contact structure on $\PP^{2n-1}=\PP(E)$, $\dim E=2n$, corresponds
to a non-degenerate skew-symmetric bilinear form on~$E$.
\end{prop}

\begin{proof}[Sketch of proof]
If a contact structure is defined by a subbundle
$\Ss\subset\T_{\PP(E)}$, put $\Ll=\T_{\PP(E)}/\Ss$; \Ll is an
invertible sheaf. The mapping
\[
(s_1,s_2)\mapsto [s_1,s_2]\bmod \Ss,
\]
where $s_1$ and $s_2$ are local sections of $\T_{\PP(E)}$ and the
brackets stand for commutator of vector fields, is a sheaf
homomorphism $\Ss\otimes_{\Oo_{\PP(E)}}\Ss\to\Ll$ which factors
through $\bigwedge^2\Ss$; one of the equivalent definitions of contact
structures is that \Ss is a contact structure if and only if the
resulting homomorphism $\psi\colon
\bigwedge^2\Ss\to\T_{\PP(E)}/\Ss=\Ll$ is a non-degenerate
skew-symmetric form on \Ss with values in \Ll. Since $\psi$ is
non-degenerate, it induces an isomorphism
\begin{equation}\label{Chern}
\Ss\xrightarrow\approx\Ss^*\otimes\Ll. 
\end{equation}
Since $c_1(\Ss)=c_1(\T_{\PP(E)})-c_1(\Ll)$ and
$c_1(\Ss^*\otimes\Ll)=-c_1(\Ss)+(2n-2)c_1(\Ll)$, the
isomorphism~\eqref{Chern} implies that $2nc_1(\Ll)=2c_1(\T_{\PP(E)})$,
whence $\Ll\cong\Oo_{\PP(E)}(2)$. The epimorphism
$\pi\colon\T_{\PP(E)}\to\Ll=\Oo_{\PP(E)}(2)$ is a section of
$\Omega^1_{\PP(E)}(2)$; it follows from the exact sequence
\[
0\to\Omega^1_{\PP(E)}(2)\to E^*\otimes\Oo_{\PP(E)}\to
\Oo_{\PP(E)}(2)\to 0
\]
(Euler's sequence twisted by $\Oo(2)$) that
$H^0(\Omega^1_{\PP(E)}(2))=\bigwedge^2E^*$. Now the element of
$\bigwedge^2E^*$ corresponding to $\pi$ is the desired skew-symmetric
form.
\end{proof}

The reader may ignore this proposition and further on, every time we
mention a contact structure on $\PP^{2n-1}$, substitute ``a contact
structure corresponding to a skew-symmetric form'' instead.

\begin{prop}\label{Leg<=>iso}
Suppose that $E$ is an even-dimensional vector space of dimension $2n$
endowed with a non-degenerate skew-symmetric bilinear form $B$ and
that $Y\subset\PP(E)$ is a projective subvariety. Then the following
two assertions are equivalent.

\textup{(1)} $Y$ is integrable with respect to the
contact structure corresponding to the form~$B$.

\textup{(2)} For any smooth point $y\in Y$, the deprojectivisation
$\widehat{T_yY}\subset E$ of the embedded tangent space $T_yY\subset
\PP(E)$ is isotropic with respect to the form~$B$.
\end{prop}

\begin{proof}
The implication $(2)\Rightarrow(1)$ is very easy. To wit, if $y=(v)$,
where $v\in F=\widehat{T_yY}$, then $B(v,w)=0$ for any $w\in F$ since
$F$ is isotropic, so $F\subset v^\perp$ and $T_yY\subset T_y\PP(E)$ is
contained in the contact hyperplane at the point~$y$.

To prove the implication $(1)\Rightarrow(2)$ denote by $\pi\colon
E\setminus\{0\}\to \PP(E)$ the natural projection. Since $\pi$ is
submersive, the pullback $\pi^*\Ss\subset\T_{E\setminus\{0\}}$ is a
subbundle of codimension~$1$, where \Ss is the contact structure
corresponding to~$B$.  If the bilinear form $B$ is defined by the
formula~\eqref{eq:B}, then the family of hyperplanes in tangent spaces
defined by the subbundle $\pi^*\Ss$ is the family of kernels of the
$1$-form $\eta=\sum_{i=0}^{2n-1}z_{2i}dz_{2i+1}$. Since $Y$ is an
integral variety of \Ss, the form $\eta$ vanishes on $\hat
Y_{\mathrm{sm}}$, whence $d\eta|_{\hat Y_{\mathrm{sm}}}=0$. Since
$d\eta=\sum_{j=0}^{2n-1}dz_{2j}\wedge dz_{2j+1}$, this vanishing
is equivalent to the assertion that tangent spaces to $Y_{\mathrm{sm}}$
are isotropic with respect to $B$. Since these tangent spaces are
deprojectivisations of embedded tangent spaces to~$Y$, we are done.
\end{proof}

\begin{prop}\label{degenerate=>cone}
Suppose that $\PP^{2n-1}=\PP(E)$ is endowed with the contact structure
corresponding to a non-degenerate skew-symmetric bilinear form~$B$
on~$E$ and that $Y\subset\PP^{2n-1}$ is a Legendrian projective
subvariety with respect to this contact structure. If $Y$ is contained
in a hyperplane in $\PP^{2n-1}$, then $Y$ is a cone over a variety
$Y_1\subset \PP^{2n-3}=\PP(E_1)$, where $E_1\subset E$ is a linear
subspace of codimension~$2$; besides, the restriction of the form $B$
to $E_1$ is non-degenerate and the variety $Y_1\subset\PP(E_1)$ is
Legendrian with respect to the contact structure
corresponding to restriction~$B|_{E_1}$.
\end{prop}

\begin{proof}
Suppose that a Legendrian (with respect to $B$) projective subvariety
$Y\subset \PP(E)$ lies in a hyperplane $H\subset \PP(E)$. One has
$H=\PP(v^\perp)$ for some $v\in E\setminus\{0\}$. Since the form $B$
is non-degenerate, there exists a $(2n-2)$-dimensional linear subspace
$E_1\subset v^\perp$ and a vector $w\in E_1^\perp$ such that
$B(v,w)\ne 0$ and $E=E_1\oplus \langle v,w\rangle$ is a
skew-orthogonal direct sum; the restriction of $B$ to $E_1$ is again
non-degenerate.

Put $p=(v)\in\PP(E)$ and denote by $\pi_p\colon H\dasharrow
\PP(E_1)$the projection from $p$. If $L\subset H$ is a linear subspace
such that the deprojectivisation $\hat L$ is isotropic, then the
deprojectivisation of $\pi_p(L)$ in $\PP(E_1)$ is also isotropic. Put
$Y_1=\pi_p(Y)$.  It follows from
Proposition~\ref{Leg<=>iso} that deprojectivisations of tangent spaces
to smooth points of~$Y$ are isotropic; now Sard's theorem together
with the above observation implies that deprojectivisations of tangent
spaces at almost all points of $Y_1$ are also isotropic.

If $\dim Y_1=\dim Y=n-1$, we obtain a contradiction since in that case
dimension of these deprojectivisations is $n>2(n-1)/2$. Thus, $\dim
Y_1=\dim Y-1=n-1$ and $Y$ is a cone over $Y_1$ with
vertex~$p$. Finally, since (deprojectivisations of) tangent spaces to
$Y_1$ are isotropic, the subvariety $Y_1\subset\PP(E_1)$ is Legendrian
with respect to the restriction of $B$ to~$E_1$.
\end{proof}

\begin{cor}\label{deg=>line}
Suppose that $C$ is a projective Legendrian curve in $\PP^3$ with a
contact structure. If $C$ is contained in a plane, then $C$ is a line.\qed
\end{cor}

If $X$ is a smooth variety, one can define a canonical contact
structure on $V=\PP^*(\T_X)$ as follows. If $p=(x,H)\in V$, where
$x\in X$ and $H\subset\T_xX$ is a hyperplane, then
$\Ss_p=\pi_*^{-1}(H)\subset \T_pV$, where $\pi\colon V\to X$ is the
projection and $\pi_*$ is the derivative of~$\pi$.

If $Y\subset X$ is a subvariety, then it is easy to check that
\begin{equation*}
\Pc_Y=\overline{\{(y,H)\in V=\PP^*(\T_X)\colon y\in Y_{\mathrm{smooth}},
  H\supset \T_yY\}}
\end{equation*}
is a Legendrian subvariety of $V$. It follows from Sard's theorem that
any Legendrian subvariety of $V$ has the form $\Pc_Y$ for some
$Y\subset X$ (see~\cite{Kleiman}). We will say that $\Pc_Y$ is the
\emph{conormal variety} of $Y$.

We will be using the above construction for $X=\PP^n$. In this
situation $\PP^*(\T_{\PP^n})$ is just the incidence relation:
\[
\PP^*(\T_{\PP^n})=\{(x,H)\in \PP^n\times(\PP^n)^*\colon x\in H\}.
\]
In coordinates, if $(x_0:\ldots:x_n)$ are homogeneous coordinates in
$\PP^n$ and $(y_0:\ldots:y_n)$ are the dual homogeneous coordinates in
$(\PP^n)^*$, one has
\begin{equation}\label{eq:PTP}
P^*(\T_{\PP^n})=\{((x_0:\ldots:x_n),(y_0:\ldots:y_n))\in
\PP^n\times(\PP^n)^*\colon \sum x_iy_i=0\}.
\end{equation}
Thus, $\PP^*(\T_{\PP^n})$ is a hyperplane section of Segre variety
$\PP^n\times\PP^n\subset \PP^{n^2+2n}$. When referring to
$\PP^*(\T_{\PP^n})$ as projective variety we will always mean this
embedding $\PP^*(\T_{\PP^n})\hookrightarrow \PP^{n^2+2n-1}$.

The following result is due essentially to R.~Bryant, at least for
$n=2$ (see~\cite[proof of Theorem F]{Bryant}).

\begin{prop}\label{prop:Bryant}
Suppose that $\PP^{2n-1}$, with homogeneous coordinates
$(z_0:\ldots:z_{2n-1})$, is endowed with the contact structure
corresponding to the skew-symmetric form~\eqref{eq:B} and that
$\PP^*(\T_{\PP^n})$ defined by the equation~\eqref{eq:PTP} is endowed
with the canonical contact structure. Then the rational mapping $\vartheta\colon
\PP^*(\T\PP^n)\dasharrow \PP^{2n-1}$ defined by the formula
\[
\vartheta \colon((x_0:\ldots:x_n),(y_0:\ldots:y_n))
\mapsto(z_0:\ldots:z_{2n-1}),
\]
where
\begin{equation}\label{eq:beta^(-1)}
\begin{aligned}
z_0&=x_0y_1,\\
z_1&=\frac12(x_1y_1-x_0y_0),\\
z_{2k-2}&=x_ky_1,\quad 2\le k\le n,\\
z_{2k-1}&=-\frac12x_0y_k,\quad 2\le k\le n,
\end{aligned}
\end{equation}
is a birational isomorphism that agrees with the named contact
structures. This birational isomorphism induces an isomorphism between
the Zariski open subsets
\begin{equation}\label{eq:def:V}
V=\{((x_0:\ldots:x_n),(y_0:\ldots:y_n))\colon x_0\ne 0, y_1\ne
0\}\subset\PP^*(\T_{\PP^n}) 
\end{equation}
and 
\[
W=\{(z_0: z_1:\ldots: z_{2n-1})\colon z_0\ne 0\}\subset\PP^{2n-1}.
\]
\end{prop}

\begin{proof}
A straightforward check shows that the rational mapping $\beta\colon
\PP^{2n-1}\dasharrow \PP^*(\T_{\PP^n})$ defined by the formulas
\[
\begin{aligned}
x_0&=z_0^2, &y_0&=-z_0z_1+\sum_{j=1}^{n-1}z_{2j}z_{2j+1},\\
x_1&=z_0z_1+\sum_{j=1}^{n-1}z_{2j}z_{2j+1}, & y_1&=z_0^2,\\
x_k&=z_0z_{2k-2}, & y_k&=-2z_0z_{2k-1},\quad 2\le k\le n
\end{aligned}
\]
is inverse to $\vartheta$ and that $\vartheta$ performs an isomorphism
between $V$ and $W$. To check that $\vartheta$ agrees with the contact
structures it suffices to show that its inverse $\beta$ agrees with contact
forms on some non-empty Zariski open set. Put 
\[
V_1=\{((x_0:\ldots:x_n),(y_0:\ldots:y_n))\in
\PP^n\times(\PP^n)^*\colon x_0\ne 0,\ y_0\ne 0,\ y_1\ne 0\};
\]
we may and will assume that $x_0=y_0=1$ on $V_1\subset V$. For each $j$,
$2\leqslant j\leqslant n$, put $\xi_j=y_j/y_1$. Then $(\lst
xn,\lst[2]\xi n)$ are local coordinates on $V_1$, and
it is easy to see that in these $(x,\xi)$ coordinates on $V_1$
the canonical contact structure on $\PP^*(\T_{\PP^n})$ may be defined
as the family of kernels of the form
\begin{equation}\label{eq:eta}
\eta=dx_1+\xi_2dx_2+\ldots +\xi_ndx_n.
\end{equation}
An immediate check shows that $\beta^*\eta=\omega$, where $\eta$ is defined
by~\eqref{eq:eta} and $\omega$ is defined by~\eqref{eq:omega}, so we
are done.
\end{proof}

\begin{prop}\label{prop:beta^(-1)}
The birational isomorphism $\vartheta \colon \PP^*(\T_{\PP^n})\to
\PP^{2n-1}$ is induced by a projection $\pi_L\colon
\PP^{n^2+2n-1}\dasharrow \PP^{2n-1}$, where $L\subset \PP^{n^2+2n-1}$
is a linear subspace of dimension $n^2-1$. The intersection $L\cap
\PP^{n^2+2n-1}$ has the form
\[
L\cap\PP^{n^2+2n-1}=\{(x,H)\in\PP^n\times(\PP^n)^*\colon
x\in H_0,\ H\ni x_0,\ x\in H\},
\]
where $H_0\subset\PP^n$ is the hyperplane defined by the
equation~$x_0=0$ and $p_0\in H_0$ is the point with homogeneous
coordinates $(0:1:0:\ldots:0)$.
\end{prop}

\begin{proof}
Proposition~\ref{prop:Bryant} shows that $\vartheta $ is defined by
bihomogeneous in $x$'s and $y$'s polynomials of bidegree~$(1,1)$, so
it is a projection of (a hyperplane section of) Serge variety, with
center~$L$ of dimension $(n^2+2n-1)-(2n-1)-1$. The intersection $L\cap
\PP^*(\T_{\PP^n})$ is the set of points where all the polynomials in
the right-hand sides of~\eqref{eq:beta^(-1)} vanish; it is easy to see
that this happens if and only if $x_0=y_1=0$, which implies the proposition.
\end{proof}

\section{Construction of self-dual varieties}\label{sec:surfaces} 

Throughout this section, $\PP^{2n-1}=\PP(E)$, where the
$2n$-dimensional linear space $E$ is endowed with a non-degenerate
skew-symmetric bilinear form~$B$. By contact structure on $\PP^n$ we
will mean the contact structure associated with~$B$. If $p=(v)\in
\PP^{2n-1}$, where $v\in E\setminus\{0\}$, then by
$p^\perp\subset\PP^{2n-1}$ we mean $\PP((v)^\perp)$. Obviously, $p$ is
contained in the hyperplane $p^\perp$.

\begin{prop}\label{Leg=>osc}
If $X\subset \PP^{2n-1}$ is an integral subvariety with respect to a
contact structure, then $\Osc_p^2X\subset p^\perp$ for any smooth
$p\in X$.
\end{prop}

\begin{proof}
Put $\dim X=d$. Suppose that $p\in X$ is a smooth point and \lst xn
are local analytic coordinates near $x$. Locally (in the classical
topology) near $p$ the variety $X\subset\PP^n$ can be parametrized by
the formula
\[
(\lst xd)\mapsto(v(\lst xd)),
\] 
where $v\colon U\to E\setminus\{0\}$ is a holomorphic immersion and
$U\subset \C^d$ is an open set; the subspace $\Osc^k_pX\subset\PP(E)$ is
projectivisation of the linear space spanned by $v$ an all its partial
derivatives up to the order~$k$. 

Since $X$ is an integral variety of the contact structure associated
with $B$, one has
\begin{equation}\label{eq:4.1}
 B(x,\partial v/\partial x_i)=0,\quad 1\le i\le d
\end{equation}
by definition of the contact structure corresponding to~$B$ and
\begin{equation}\label{eq:4.2}
 B(\partial v/\partial x_i,\partial v/\partial x_j)=0,\quad 1\le
 i,j\le d  
\end{equation}
by Proposition~\ref{Leg<=>iso}. Differentiating~\eqref{eq:4.1} with
respect to $x_i$, one has
\[
B(\partial v/\partial x_i,\partial v/\partial x_i)+B(v,\partial^2
v/\partial x_i^2)=B(v,\partial^2
v/\partial x_i^2)=0,
\]
so $\partial^2 v/\partial x_i^2\in
(v)^\perp$. Differentiating~\eqref{eq:4.1} with respect to $x_j$,
$j\ne i$, one has
\[
B(\partial v/\partial x_j,\partial v/\partial
x_i)+B(v,\partial^2v/\partial x_i\partial
x_j)=B(v,\partial^2v/\partial x_i\partial x_j)=0
\]
(the first summand vanishes by virtue of~\eqref{eq:4.2}), so
$\partial^2v/\partial x_i\partial x_j\in (v)^\perp$. Thus,
$\Osc_p^2X\subset p^\perp$ as required.
\end{proof}

\begin{prop}\label{cor:leg=>self-dual}
Suppose that $X\subset\PP^{2n-1}$ is a Legendrian subvariety with
respect to a contact structure. If $\dim\Osc_p^2X=2n-2$ for
general $p\in X$ and $X$ is not contained in a hyperplane, then $X$ is
osculating self-dual.
\end{prop}

\begin{proof}
Proposition~\ref{Leg=>osc} implies that $\Osc_p^2X\subset p^\perp$;
since dimensions are the same, these linear spaces coincide. Now the
desired linear isomorphism $\PP^n\to(\PP^n)^*$ that maps $X$ to
$X^\vee$ is the one induced by the isomorphism $E\to E^*$
corresponding to the bilinear form~$B$.
\end{proof}

\begin{lemma}\label{P_X:non-deg}
Suppose that $X\subset\PP^n$ is an irreducible subvariety. If $X$ is
not contained in a hyperplane and $X$ is not a cone then the conormal
variety $\Pc_X\subset\PP^*(\T_{\PP^n})$ has a non-empty intersection
with the Zariski open subset $V\subset \PP^*(\T_{\PP^n})$ defined
in~\eqref{eq:def:V}.

In particular, if $n=2$ then $\Pc_X\cap V\ne\varnothing$ provided
that $X$ is not a line.
\end{lemma}

\begin{proof}
Since $X$ is not contained in a hyperplane,
$X\cap\Ab^n=\{(1:x_1:\ldots:x_n)\}\ne\varnothing$. Thus, to check that
$\Pc_X\cap V\ne\varnothing$ it suffices to check that the coordinate
$y_1$ is not identically zero on $\Pc_Y$. Assume the converse; then
all the tangent hyperplanes to smooth points of $X$ pass through the
point $p=(0:1:0:\ldots:0)$, which is possible only if $X$ is a cone
with vertex $p$ (to justify this assertion, apply Sard's theorem to
the projection with center~$p$), which contradicts the hypothesis.
\end{proof}

\begin{prop}\label{prop:constr_of_selfdual}
  If $X\subset\PP^n$, $n\ge 3$, is a general hypersurface of degree at
  least~$3$ and $\Pc_X\subset\PP^*(\T_{\PP^n})$ is its conormal
  variety, then the proper image
  $C=\vartheta (\Pc_X)\subset\PP^{2n-1}$, where $\vartheta\colon\PP^*(\T_{\PP^n})\to
  \PP^{2n-1}$ is the birational isomorphism
  defined in Proposition~\ref{prop:Bryant}, is an $(n-1)$-dimensional
  subvariety such that $\dim\Osc_x^2C=2n-2$ for general $x\in C$, $C$
  is not contained in a hyperplane, and $C^\vee$ is projectively
  equivalent to $C$.
\end{prop}

\begin{proof}
Since we may assume that $X$ is not a cone, Lemma~\ref{P_X:non-deg}
shows that the image $C=\vartheta (\Pc_X)$ is well defined.  Since the
birational isomorphism $\vartheta$ agrees with the contact structures, the
variety $C$ is Legendrian with respect to a contact structure on
$\PP^{2n-1}$. Now Proposition~\ref{cor:leg=>self-dual} shows that to
prove the proposition it suffices to check that, for general $X$, the
variety $C$ is not contained in a hyperplane and $\dim\Osc_x^2C=2n-2$
for general $x\in C$.

Since dimensions of osculating spaces are lower semicontinuous and
dimension of the second osculating space to a Legendrian subvariety in
$\PP^{2n-1}$ is at most $2n-2$ (Proposition~\ref{Leg=>osc}), the
second assertion will follow once we have, for each $d\ge3$, an
example of a hypersurface $X\subset\PP^n$ of degree $d$ such that
the general osculating space to $\vartheta(\Pc_X)$ has dimension~$2n-2$. Let
us look for such examples among hypersurfaces with the equation
$x_0^{d-1}x_1+F(\lst[2]xn)=0$, where $(x_0:\ldots:x_n)$ are
homogeneous coordinates in $\PP^n$ and $F$ is a homogeneous polynomial
of degree~$d$. On the affine open subset $\{x_0=1\}$ this hypersurface
has equation $x_1+F(\lst[2] xn)=0$. 

It follows from Proposition~\ref{prop:Bryant} that
$\vartheta (\Pc_X)$ is closure of the set of points
$(1:z_1:\ldots:z_{2n-1})$, where
\begin{equation}\label{eq:Pc_X}
\begin{aligned}
z_1&=\frac{2-d}2F(\lst[2]xn),\\
z_{2k-2}&=x_k,\quad 2\le k\le n,\\
z_{2k-1}&=-\frac12\cdot\frac{\partial F}{\partial
  x_k},\quad 2\le k\le n
\end{aligned}
\end{equation}
(\lst[2]xn are arbitrary).

The second fundamental form is spanned by Hessians of right-hand sides
of~\eqref{eq:Pc_X}, where by Hessian of a function $\ph$ depending on
the variables \lst[2]xn we mean the quadratic form
$\sum_{i,j}(\partial^2\ph/\partial x_i\partial x_j)t_it_j$. Now if we
put $F=x_2^d+\ldots+x_n^d$, then it is easy to check that these
Hessians span an $(n-1)$-dimensional space, as required.

It remains to check that for the general $X$ the variety $C$ is not
contained in a hyperplane. To that end, we invoke
Proposition~\ref{degenerate=>cone}. According to this proposition, if
$C$ is contained in a hyperplane, then $C$ is a cone over
$C_1\subset\PP^{2n-3}\subset\PP^{2n-1}$, where $C_1$ is Legendrian
with respect to a contact structure on~$\PP^{2n-3}$. Since $C$ is a
cone over $C_1$, dimensions of the second fundamental form at the general
point are the same for $C$ and $C_1$; since $C_1$ is Legendrian in
$\PP^{2n-3}$, $\dim\Osc^2_xC_1\leqslant n-2$ by
Proposition~\ref{Leg=>osc}, so $\dim\Osc^2_xC\leqslant n-2$ as well,
but we know that this is not the case for $C=\vartheta (X)$ for
general~$X$. This contradiction completes the proof.
\end{proof}

In the case of curves in $\PP^3$ one can say a bit more.

\begin{prop}\label{constr_curves}
Suppose that $X\subset\PP^2$ is an irreducible projective curve of
degree greater than one. Then there exists an osculating self-dual curve
$C\subset\PP^3$ such that $C$ is birational to $X$ and $\deg C=\deg
X+\deg X^*$, where $X^*\subset(\PP^2)^*$ is the dual curve.
\end{prop}

\begin{proof}
Arguing as in the proof of Proposition~\ref{prop:constr_of_selfdual},
if $\Pc_X\subset\PP^*(\T_{\PP^2})$ is the conormal variety of~$X$,
then $C=\vartheta (\Pc_X)\subset\PP^3$ is Legendrian, hence self-dual,
if $C$ is not contained in a plane and $\Pc_X\cap W\ne\varnothing$
(since $\dim\Osc_p^2C=2$ automatically for general $p\in C$ if
$C\subset \PP^3$ is a curve that is not contained in a plane). Since
$X\subset\PP^2$ is not a line, it is not a cone, so
Lemma~\ref{P_X:non-deg} ensures that $\Pc_X\cap W\ne\varnothing$ and
$C=\vartheta (\Pc_X)$ is well defined. To be able to control $\deg C$,
recall that, according to Proposition~\ref{prop:beta^(-1)}, the
rational mapping $\vartheta $ is induced by the projection
$\pi_L\colon \PP^7\dasharrow\PP^3$, where $L\subset\PP^7$ is a
$3$-dimensional linear space such that
\[
L\cap \PP^*(\T_{\PP^3})=\{(x,H)\in\PP^2\times(\PP^2)^*\colon
x\in H_0,\ H\ni x_0,\ x\in H\};
\]
here, $H_0\subset\PP^2$ is a line and $p_0\in H_0$ is a point. So,
this intersection is the union of two lines, $\{(x,H_0)\colon x\in
H_0\}$ and $\{(x_0,H)\colon H\ni x_0\}$; these lines intersect at the
point $(x_0,H_0)$.  It is clear that $L\cap\Pc_X$ is the set of
couples $(x,H)$ such that either $H=H_0$ is tangent to $X$ at the
point $x$ or $H$ is tangent to $X$ at the point~$x=x_0$ (if $x\in X$
is singular, we say that a line~$H$ is tangent to $C$ at $X$ if it is
a limit of tangents at smooth points tending to~$x$). Thus, if $X$ is
not tangent to $H_0$ (i.e., $H_0\notin X^*$) and $x_0\notin X$, then
$\Pc_X\cap L=\varnothing$
since the restriction of the projection $\pi_L=\vartheta $ to the subset
$V=\{((x_0:x_1:x_2),(y_0:y_1:y_2))\in\PP^*(\T_{\PP^2})\colon
x_0\ne1,\ y_1\ne 1\}$
is
an isomorphism onto its image (see Proposition~\ref{prop:Bryant}),
this implies that $\deg C=\deg \Pc_X$, where in the right-hand side we
regard $\Pc_X$ as a curve in the $\PP^7$ in which $\PP^*(\T_{\PP^2})$
is embedded. Denoting the projections of $\PP^*(\T_{\PP^2})$ on
$\PP^2$ and $(\PP^2)^*$ by $\pr_1$ and $\pr_2$ respectively, one sees
that
$\Oo_{\PP^*(\T_{\PP^2})}(1)=\pr_1*\Oo_{\PP^2}(1)\otimes\pr_2^*\Oo_{(\PP^2)^*}(1)$.
So, if $L\cap\Pc_X=\varnothing$, then 
\[
\deg C=\deg\Oo_{\PP^*(\T_{\PP^2})}(1)|_{\Pc_X}=\deg X+\deg X^*.
\]
It remains to observe that for any curve $X\subset\PP^2$ there exists
a curve $X_1\subset\PP^2$ such that $X_1$ is projectively equivalent
to $X$ and $X_1$ is not tangent to $H_0$ and does not pass through
$p_0$. Putting $C=\vartheta (\Pc_{X_0})$ one obtains the required self-dual curve.
\end{proof}

\begin{prop}\label{prop:smooth_curve}
Any smooth projective curve is isomorphic to a curve in $\PP^3$ that
is Legendrian with respect to a contact structure. 

In particular, any smooth projective curve can be embedded in $\PP^3$
as an osculating self-dual curve. 
\end{prop}

We begin the proof with two lemmas. In these lemmas we assume that
homogeneous coordinates on $\PP^3$ are $(z_0:z_1:z_2:z_3)$,
homogeneous coordinates on $\PP^2$ are $(x_0:x_1:x_2)$ and dual
homogeneous coordinates on $(\PP^2)^*$ are $(y_0:y_1:y_2)$. By tangent
line to a plane nodal curve~$X$ we mean a line that is tangent either
to $C$ at a smooth point or to a branch of $C$ at a node. By
$\vartheta\colon \PP^*(\T_{\PP^2})\subset\PP^3$ we mean the birational
morphism defined in Proposition~\ref{prop:Bryant}.

\begin{lemma}\label{lemmaA}
Suppose that $X\subset\PP^2$ is a nodal curve with the following
properties.

\textup{(1)} $X$ intersects transversally the line $\{x_0=0\}$.

\textup{(2)} The lines tangent to $X$ at its intersection points with the
line $\{x_0=0\}$, do not pass through the point $(0:1:0)$.

\textup{(3)} The lines tangent to $X$ at inflection points \textup(including
inflection points of branches at nodes, if such nodes exist\textup) do not pass
through the point $(0:1:0)$.

\textup{(4)} The curve $X$ does not pass through the point $(0:0:1)$.

Then the conormal variety $\Pc_X\subset\PP^*\T_{\PP^3}$ is smooth,
lies in the open set where the birational isomorphism $\vartheta$ is
defined, and the mapping $\vartheta |_{\Pc_X}\colon \Pc_X\to
\PP^3$ is an immersion.
\end{lemma}

\begin{proof}
Since $X$ is nodal, $\Pc_X$ is smooth, so we are only to check that
$\vartheta $ is defined on $\Pc_X$ and that the derivative of the
restriction $\vartheta|_{\Pc_X}$ does not vanish.

The first assertion follows from Proposition~\ref{prop:beta^(-1)} and
hypothesis~(2). 

To check that $\vartheta|_{\Pc_X}$ is an immersion observe that the
restriction of $\vartheta$ to the subset $V=\{((x_0:x_1:x_2),
(y_0:y_1:y_2))\in \PP^*(\T_{\PP^2})\colon x_0\ne 0,\ y_0\ne 0\}$ is an
isomorphism onto its image. Thus, it suffices to check that
$\vartheta|_{\Pc_X}$ is an immersion at the points of $\Pc_X$ for
which either $x_0=0$ or $y_1=0$ (these coordinates cannot both vanish
for a point of $\Pc_X$ because of hypothesis~(2)).

\emph{Case 1. Points of $\Pc_X$ for which $x_0\ne 0$, $y_1=0$.}  If
$x_0\ne0$ and $y_1=0$ for a point $(p,H)\in\Pc_X$, then the curve $X$
(which is smooth at $p$ thanks to hypothesis~(1)) can be analytically
parametrized near $p$ by the formula $\gamma\colon t\mapsto
(1:a(t):b(t))$, where $\gamma(0)=p$ and $b'(0)=0$. Since the tangent
line at $p(t)$ has equation
\[
(a'(t)b(t)-b'(t)a(t))x_0+b'(t)x_1-a'(t)x_2=0,
\]
the curve $\Pc_X$ can be parametrized near $(p,H)$ as
\[
t\mapsto((1:a(t):b(t)),(a'(t)b(t)-b'(t)a(t):b'(t):-a'(t));
\]
using formula~\eqref{eq:beta^(-1)} we see that the curve
$\vartheta(\Pc_X)$ near the point $(p,H)$ can be parametrized by
\[
v\colon t\mapsto (b'(t):(2a(t)b'(t)-a'(t)b(t))/2:
b(t)b'(t):a'(t)/2).
\]
The mapping $v$ is not an immersion exactly at the points where $v$ and $v'$
are proportional; taking into account the equation $b'(0)=0$, one has
\begin{equation}\label{eq:v,v':case1}
v'(0)=(b''(0):*:*:*),\ v(0)=(0:*:*:a'(0)/2),
\end{equation}
where stars stand for irrelevant terms. It follows
from~\eqref{eq:v,v':case1} that if
$v(0)$ and $v'(0)$ are
proportional then $a'(0)b''(0)=0$. However, $a'(0)\ne 0$ since
$b'(0)=0$. Thus, $b''(0)=0$, so $p$ is an inflexion point of the curve
$X$, and the tangent to $X$ at $p$ passes through $(0:1:0)$ since
$b'(0)=0$, which contradicts hypothesis~(3). Thus,
$\vartheta|_{\Pc_X}$ is an immersion at~$(p,H)$.

\emph{Case 2. Points of $\Pc_X$ for which $x_0=0$, $y_1\ne0$.}  If
$x_0\ne0$ and $y_1=0$ for a point $(p,H)\in\Pc_X$, then $x_1\ne 0$ for
the point $p$ thanks to hypothesis~(4). So, near the point~$p$ the
curve $X$ (or, if $p$ is a node, the branch to which $H$ is tangent)
can be parametrized as $t\mapsto (a(t):1:b(t))$, where $a(0)=0$.  A
computation similar to what we did in Case~1 shows that $\Pc_X$ near
the point $(p,H)$ can be parametrized as
\[
t\mapsto((a(t):1:b(t)),(b'(t):a(t)'b(t)-a(t)b'(t):-a'(t)))
\]
and the curve $\vartheta(\Pc_X)$ near the point $(p,H)$ can be
parametrized as
\begin{multline*}
v\colon t\mapsto
(a(t)a'(t)b(t)-(a(t))^2b'(t):(a'(t)b(t)-2a(t)b'(t))/2:{}\\
a'(t)b(t)^2-a(t)b(t)b'(t):a(t)a'(t)/2),
\end{multline*}
Again $v$ fails to be immersion where $v$ and $v'$ are
proportional. Taking into account the equation $a(0)=0$, one has
\[
v(0)=(0:a'(0)b(0)/2:*:*),\
v'(0)=((a'(0))^2b(0):*:*:*),
\]
so if $v(0)$ and $v'(0)$ are proportional then
$(a'(0))^3(b(0))^2=0$. However, $a'(0)\ne 0$ since $X$ is transversal
to the line $\{x_0=0\}$ (hypothesis~(1)). Hence, $b(0)=a(0)=0$ and the
curve $X$ passes through the point~$(0:0:1)$. This contradicts
hypothesis~(4).
\end{proof}

\begin{lemma}\label{lemmaB}
Suppose that $X\subset\PP^2$ is a nodal curve with the following
properties.

\textup{(1)} $X$ intersects transversally the line $\{x_0=0\}$.

\textup{(2)} Tangent lines to $X$ at its intersection points with the
line $\{x_0=0\}$, do not pass through the point $(0:1:0)$.

\textup{(3)} No bitangent to $X$ passes through the point~$(0:1:0)$.

Then the conormal variety $\Pc_X\subset\PP^*\T_{\PP^3}$ is smooth,
lies in the open set where the birational isomorphism
$\vartheta (\Pc_X)\subset\PP^3$ is defined, and the restriction
$\vartheta |_{\Pc_X}\colon \Pc_X\to \PP^3$ is one-to-one onto its image.
\end{lemma}

\begin{proof}
The first two assertions follow from hypotheses~(1) and~(2) as before,
and again $\Pc_X$ is smooth.  To prove injectivity of
$\vartheta|_{\Pc_X}$, represent $\Pc_X$ as disjoint union of the
following three subsets $A$, $B$ and~$C$:
\[
\begin{aligned}
A&=\{((x_0:x_1:x_2),(y_0:y_1:y_2))\in\Pc_X\colon x_0=0\},\\
B&=\{((x_0:x_1:x_2),(y_0:y_1:y_2))\in\Pc_X\colon x_0\ne0,\ y_1=0\},\\
C&=\{((x_0:x_1:x_2),(y_0:y_1:y_2))\in\Pc_X\colon x_0\ne0,\ y_1\ne0\}.
\end{aligned}
\]
Now the required injectivity of $\vartheta |_{\Pc_X}$ is implied by
the following chain of assertions.

1. \emph{$\vartheta $ is injective on $A$}. Indeed, if
$p=((0:x_1:x_2),(y_0:y_1:y_2))\in A$, then
$\vartheta (p)=(0:\frac12x_1:x_2:0)$. Thus, if $p_1,p_2\in A$, then
$\vartheta (p_1)=\vartheta (p_2)$ if an only if $\pi(p_1)=\pi(p_2)$,
where $\pi\colon \Pc_X\to X$ is the natural projection. Since the line
$\{x_0=0\}$ does not pass through the nodes of $X$ (hypothesis~(1)), this
implies that $p_1=p_2$.

2. $\vartheta (A)\cap \vartheta (B\cup C)=\varnothing$. Indeed,
the $z_0$ coordinate is zero for any point in $\vartheta (A)$, and
$z_0$ is non-zero for any point in $\vartheta (C)$,
see~\eqref{eq:beta^(-1)}. Thus, $\vartheta (A)$ is disjoint
with $\vartheta (C)$. If $\vartheta (p)=\vartheta (q)$, where
$p=((0:x_1:x_2),(\cdot:\cdot:\cdot))\in A$ and $q=((\cdot:\cdot:\cdot),
(y_0:0:y_2))\in B$, then 
\[
(0:\frac12x_1:x_2:0)=(0:y_0:0:y_2),
\]
whence $y_2=0$. Thus, the second component of the point $q\in\Pc_X$ is
the line with homogeneous coordinates $(1:0:0)$, i.e., the line with
equation $x_0=0$. This is, however, impossible since this line is not
tangent to $X$, thanks to hypothesis~(1).

3. \emph{$\vartheta $ is injective on $B$}. Indeed, if  
$p=((x_0:x_1:x_2),(y_0:0:y_2))\in B$, then
$\vartheta (p)=(0:y_0:0:y_2)$. So, if the points $p_1=(x_1,H_1)$
and $p_2=(x_2,H_2)$ ($x_i\in X$, $H_i$ is tangent to $X$ at
$x_i$) lie in $B$ and if $\vartheta (p_1)=\vartheta (p_2)$, then
$\ell_1=\ell_2$ and the line $H=H_1=H_2$ passes through the
point $(0:1:0)$ and is tangent to $X$ at two different points $x_1$
and $x_2$; this contradicts hypothesis~(3).

4. $\vartheta (B)\cap \vartheta (C)=\varnothing$. Indeed the
$z_0$ coordinate of $\vartheta (p)$ is zero if $p\in B$ and it is
non-zero if $p\in C$.

5. \emph{$\vartheta$ is injective on $C$}. 
This is implied by Proposition~\ref{prop:Bryant}.
\end{proof}

\begin{proof}[Proof of Proposition~\ref{prop:smooth_curve}]
Any smooth projective curve $C\subset\PP^N$ can be birationally
projected to a nodal plane curve $X\subset\PP^2$; it is clear that
$C$, being the normalization of $X$, is isomorphic to $\Pc_X$. For a
general projective transformation $A\colon \PP^2\to\PP^2$, the curve
$X_1=AX\subset\PP^2$ satisfies the hypotheses of Lemmas~\ref{lemmaA}
and~\ref{lemmaB}, so the curve $\vartheta (\Pc_{X_1})\subset\PP^3$ is
smooth, isomorphic to~$C$, and Legendrian with respect to the contact
structure defined by the formula~\eqref{eq:omega} with $n=2$
\end{proof}

\begin{cor}[from the proof]
If $C\subset\PP^N$ is a smooth projective curve of degree $d$ and
genus $g$, then there exists a Legendrian curve $C'\subset\PP^3$ such
that $C'$ is isomorphic to $C$ and $\deg C'=3d+2g-2$. In particular,
if $C$ is a smooth plane curve of degree~$d$, then $\deg C'=d^2$.
\end{cor}

\begin{proof}
If $\deg C=d$ and genus of $C$ equals $g$, then its general projection
$X\subset\PP^2$ has $\nu=(d-1)(d-2)/2-g$ nodes, whence 
\[
\deg X^*=\deg(X')^*=d(d-1)-2\nu=2d+2g-2.
\]
In the proof of Proposition~\ref{constr_curves} we found out that
$\deg \Pc_{X'}=\deg X'+\deg(X')^*$, whence the formula.
\end{proof}

\section{Concluding remarks}\label{sec:conclusion}

Although Legendrian subvarieties in odd-dimensional projective spaces
abound, there exist osculating self-dual varieties ($k$-dimensional in
$\PP^{2k+1}$) that are not Legendrian with respect to any contact
structure on $\PP^{2k+1}$. 

For $k=1$, i.e., for the case of curves in
$\PP^3$, it is easy to produce a family of examples. 

Recall that a \emph{monomial curve} in $C_{a,b,c}\subset\PP^3$ is the
closure of the set of points with homogeneous coordinates
$(1:t^a:t^b:t^c)$, where $a$, $b$, $c$ are positive integers,
$(a,b,c)=1$, and $a<b<c$.

\begin{prop}
Any monomial curve in $\PP^3$ is osculating self-dual. The monomial
curve $C_{a,b,c}\subset\PP^3$ is Legendrian with respect to an
appropriate contact structure on $\PP^3$ if an only if the sequence of
exponents $(0,a,b,c)$ is symmetric, i.e., $(0,a,b,c)=(0,c-b,c-a,c)$.
\end{prop}

\begin{proof}
A Zariski open part of the curve $C_{a,b,c}$ can be (locally) parametrized by
the formula $t\mapsto (v(t))$, where $v(t)=(1,t^a,t^b,t^c)\in\C^4$,
$\PP^3=\PP(\C^4)$, $t\in\C$. Homogeneous coordinates of the osculating
dual curve $C^\vee$ are, up to signs, $3\times3$ minors of the matrix
\[
\begin{pmatrix}
  1&t^a&t^b&t^c\\
  0&at^{a-1}&bt^{b-1}&ct^{c-1}\\
  0&a(a-1)t^{a-2}&b(b-1)t^{b-2}&c(c-1)t^{c-2}
\end{pmatrix}
\]
(of which the rows are $v(t)$, $v'(t)$, and $v''(t)$). A simple
computation shows that $C_{a,b,c}^\vee$ is projectively equivalent to
the curve that can be locally parametrized as $t\mapsto
(1:t^{c-b}:t^{c-a}:t^c)$. After the linear automorphism that
rearranges homogeneous coordinates in reverse order and the change
of parameter $t=1/s$, this dual curves becomes $C$; this proves
self-duality.

Now the curve $C_{a,b,c}$ is Legendrian if and only if there exists a
non-degenerate skew-symmetric form~$B$ on $\C^4$ such that
$B(v(t),v'(t))=0$ identically. If matrix of this bilinear form is
$\|p_{ij}\|_{0\le i,j\le 3}$, then
\begin{multline}\label{eq:B(v,v')}
B(v(t),v'(t))=ap_{01}t^{a-1}+bp_{02}t^{b-1}+cp_{03}t^{c-1}\\
{}+(b-a)p_{12}t^{a+b-1}+(c-a)p_{13}t^{a+c-1}+(c-b)p_{23}t^{b+c-1}.
\end{multline}
If the sequence $(0,a,b,c)$ is not symmetric, then all the exponents
in the right-hand side of~\eqref{eq:B(v,v')} are different, so each
$p_{ij}$ is zero and the required contact structure does not
exist. If, on the other hand, this sequence is symmetric, i.e., if
$c=a+b$, then right-hand side of~\eqref{eq:B(v,v')} is identically
zero if and only if $cp_{03}+(b-a)p_{12}=0$, so putting
\[
B=
\begin{pmatrix}
0&0&0&a-b\\
0&0&c&0\\
0&-c&0&0\\
b-a&0&0&0
\end{pmatrix}
\]
one obtains a contact structure with respect to which the curve
$C_{a,b,c}$ is Legendrian.
\end{proof}

The following proposition provides an example in higher dimensions.

\begin{prop}
Denote by $V\subset\PP^{2k+1}$, where $k\ge2$ is an integer, the
closure of the set of points $(v(t))$, $t\in\C^k$, where
\[
v(t)=(1,\lst tk,\lst{t^2}k,t_1^3+\ldots+t_k^3).
\]
Then $\dim V=k$, $\dim_p\Osc_p^2V=2k$ for general $p\in V$, $V$ is
osculating self-dual, but $V$ is not Legendrian with respect to any
contact structure on $\PP^{2n+1}$.
\end{prop}

\begin{proof}
One has
\begin{equation}\label{eq:x^3+y^3}
\begin{matrix}
v&=(1,&t_1,&t_2,&\ldots,&t_k,&t_1^2,&t_2^2,&\ldots,&t_k^2,&t_1^3+\ldots+t_k^3),\\
\frac{\partial v}{\partial
  t_1}&=(0,&1,&0,&\ldots,&0,&2t_1,&0,&\ldots,&0,&3t_1^2),\\
\hdotsfor{11}\\
\frac{\partial v}{\partial
  t_k}&=(0,&0,&0,&\ldots,&1,&0,&0,&\ldots,&2t_k,&3t_k^2),\\
\frac{\partial^2 v}{\partial
  t_1^2}&=(0,&0,&0,&\ldots,&0,&2,&0,&\ldots,&0,&6t_1),\\
\hdotsfor{11}\\
\frac{\partial^2 v}{\partial
  t_k^2}&=(0,&0,&0,&\ldots,&0,&0,&0,&\ldots,&2,&6t_k)
\end{matrix}
\end{equation}
(other second partial derivatives of $v$ are identically zero). Thus,
for general \lst tk, dimension of the second osculating space is $2k$
indeed. Homogeneous coordinates of $V^\vee$ are parametrized by
$(2k+1)\times(2k+1)$-minors of the $(2k+1)\times(2k+2)$-matrix formed
by the right-hand sides of~\eqref{eq:x^3+y^3}. Direct computation
shows that these coordinates, up to non-zero constant factors, are 
\[
(1:t_1:\ldots:t_k:t_1^2:\ldots:t_k^2:t_1^3+\ldots+t_k^3+P(\lst tk)),
\]
where $P(\lst tk)$ is a linear combination of \lst tk and \lst{t^2}k
with constant coefficients. It is clear that this variety is
projectively equivalent to~$V$, so $V$ is osculating self-dual.

Suppose now that $V$ is Legendrian with respect to the contact
structure corresponding to a skew-symmetric form $B$ with matrix
$\|p_{ij}\|$. Proposition~\ref{Leg<=>iso} implies that $B(\partial
v/\partial t_i,\partial v/\partial t_j)=0$ identically for $1\le
i<j\le k$; substituting the expressions from~\eqref{eq:x^3+y^3}, one
obtains that $p_{ij}=p_{i,k+j}=p_{k+i,k+j}0$ for $1\le i,j\le k$ and
$p_{i,2k+1}=p_{k+i,2k+1}=0$ for $1\le i\le k$. Similarly, since
$B(v,\partial v/\partial t_j)=0$ identically for $1\le i\le k$, one
obtains, taking into account that $p_{\alpha\beta}=0$ for $1\le
\alpha, \beta\le 2k$, that $p_{0,2k+1}=0$. These vanishing implies
that $\det\|p_{ij}\|=0$, which contradicts the non-degeneracy of the
form~$B$. 
\end{proof}

For $k=2$, the surface $V\subset\PP^5$ is projectively equivalent to
Togliatti's surface~(II) (see \cite[p. 261]{Togliatti}).

\bibliographystyle{amsalpha}

\bibliography{auto}

\end{document}